\documentclass{amsart}
\usepackage{amsmath}
\usepackage{amsxtra}
\usepackage{amstext}
\usepackage{amssymb}
\usepackage{latexsym}

\usepackage{verbatim}
\usepackage{tabls}
\usepackage{rotating}
\usepackage{color}

\def\Om{\Omega}

\def\om{\omega}

\def\R{\mathbf R}

\numberwithin{equation}{section}

\newtheorem{theorem}{Theorem}[section]

\newtheorem{corollary}[theorem]{Corollary}

\theoremstyle{definition}

\newtheorem{remark}[theorem]{Remark}

\title[The Hessian Sobolev inequality] {The Hessian Sobolev inequality  and its extensions}

\author[Igor E. Verbitsky]{Igor E. Verbitsky}
 \address{Department of Mathematics, University of Missouri, Columbia, MO 65211}



 \email{verbitskyi@missouri.edu}

\thanks{Supported in part by NSF grants DMS-0556309 and DMS-1161622}

\begin{document}

\maketitle


\medskip

{\footnotesize


} 







\begin{abstract}
The  Hessian Sobolev inequality of X.-J. Wang, and the Hessian Poincar\'e inequalities of Trudinger and Wang  are fundamental to  differential and conformal geometry, and geometric PDE.  These remarkable inequalities were originally established via gradient flow methods. In this paper, direct elliptic proofs are given, and  extensions to trace inequalities with general measures in place of Lebesgue measure are obtained. The new techniques rely on global estimates of solutions to  Hessian equations in terms of Wolff's potentials, and  duality arguments making use of a non-commutative  inner product on the cone of $k$-convex functions.
\end{abstract}


\section{Introduction}\label{Introduction}

Let $F_k$  ($k=1, 2, \ldots, n$) be the $k$-Hessian operator defined by
 \begin{equation}\label{I01}
 F_k[u] =\sum_{1\leq i_{1}<\cdots<i_{k}\leq n}\lambda_{i_{1}}\cdots\lambda_{i_{k}},
  \end{equation}
where $ \lambda_{1},\ldots,\lambda_{n}$ are the eigenvalues of the
Hessian matrix $D^{2}u$ on $\R^n$. In other words, $F_{k}[u]$ is the sum of
the $k\times k$ principal minors of $D^{2}u$, which coincides with
the Laplacian $F_1[u] = \Delta u$ if $k=1$; the operator $F_2[u] = \frac 1 2 \left [ (\Delta u)^2 -  |D^2 u|^2 \right ]$ if $k=2$; and the Monge--Amp\`ere
operator $F_n[u] = {\rm det} \, (D^2 u)$ if $k=n$.

The basic existence and regularity theory  for fully
nonlinear equations of Monge--Amp\`ere type which involve the $k$-Hessian
$F_k[u]$ is presented in  \cite{CNS}. In \cite{TW2}--\cite{TW3}, Trudinger and Wang introduced the concept of
the Hessian measure $\mu[u]$ associated with  $F_k [u]$,
and established the fundamental weak continuity theorem. Using these tools, Labutin \cite{L} obtained local pointwise estimates for solutions of Hessian equations in terms of nonlinear Wolff's potentials, and used them to deduce an analogue of the Wiener criterion. The corresponding global Wolff's potential estimates  were recently derived by N. C. Phuc and the author, and applied to  Hessian  equations of Lane--Emden type \cite{PV1}, \cite{PV2}. These estimates are used extensively below, along with
the underlying existence theory for Hessian and quotient Hessian equations \cite{CNS}, \cite{T1}.

The Hessian Sobolev inequality  states that, for
$k$-convex $C^2$-functions $u$ with zero boundary values,
\begin{equation}\label{h-s-in}
||u||_{L^{q} (\Omega)} \le C \, \Big (\int_\Omega |u| \, F_k[u] \, dx \Big)^{\frac 1{k+1}},
 \end{equation}
where $q = \frac{ n(k+1)}{n-2k}$, $1\le k <\frac{n}{2}$, and $\Omega$ is a $(k-1)$-convex domain with
$C^2$ boundary in $\R^n$.  (The corresponding best constant $C$ and
radial minimizers $u$  are known for $\Omega=\R^n$.) When $k=1$ and $q=
\frac{2n}{n-2}$, this is the classical Sobolev inequality. The case $k=\frac{n}{2}$, for even $n$,  requires the usual modifications involving exponential integrability.

This fully nonlinear inequality  was originally established by X.-J. Wang \cite{W} using the gradient flow method and estimates of solutions
of parabolic Hessian equations, with
a good control of the constants as $t \to \infty$. It effectively reduces the problem to radially symmetric functions in a ball.
Earlier work on Monge--Amp\`ere type integrals and their extensions is due to K. Tso \cite{Ts1}, \cite{Ts2} for convex functions.

The Hessian Sobolev inequality is central to the theory of Hessian equations and $k$-convexity,
differential  and conformal geometry, in particular, the Yamabe problem for higher order curvatures ( \cite{STW}, \cite{T2}, \cite{TW3}, and the literature cited there).

We develop a direct elliptic approach to (\ref{h-s-in}) which leads to significant extensions  where the lack of symmetry makes a  reduction to the radial setup  infeasible. It is also applicable to  the Hessian Poincar\'e inequalities
\begin{equation}\label{h-s-poin0}
|| u ||_{L^{1} (\Omega)} \le C \, \Big (\int_\Omega |u| \, F_k[u] \, dx \Big)^{\frac 1{k+1}},
 \end{equation}
\begin{equation}\label{h-s-poin}
|| D u ||_{L^{2} (\Om)} \le C \, \Big (\int_\Om |u| \, F_k[u] \, dx \Big)^{\frac 1{k+1}}.
 \end{equation}
The preceding inequalities, along with their higher order analogues, were obtained by Trudinger and Wang \cite{TW1}  using the Hessian
flow method, which gives the best constants for individual $(k-1)$-convex domains $\Om$.

In Sec.~\ref{hess-duality-sec}, we will give transparent proofs of   (\ref{h-s-poin0})  with the
best constant, and  (\ref{h-s-poin}), as well as its higher order
versions, with explicit constants, using a nonlinear duality argument discussed below.

The Hessian Sobolev inequality (\ref{h-s-in}) is more complex than Poincar\'e inequalities, which is manifested
by the presence of the critical exponent $q=\frac{ n(k+1)}{n-2k}$,
$k = 1, 2, \ldots, [\frac n 2]$. Our proofs
of  (\ref{h-s-in}) and its extensions are based in part on the idea of the integral representation, which is reminiscent of Sobolev's original proof  of the classical Sobolev inequality in the linear case. For Hessian functionals and $k$-convex functions,  integral representations are provided by means of Th. Wolff's potentials,
\begin{equation}\label{h-s-wolffp}
{\rm\bf W}_{\alpha,\, p}^R \mu (x)=\int_{0}^{R}\Big[\frac{\mu(B_{t}(x))}{t^{n-\alpha p}
}\Big]
^{\frac{1}{p-1}}\frac{dt}{t}, \qquad x \in \Om,
 \end{equation}
where $\mu$ is a positive Borel measure on $\Om$, $\alpha = \frac{2k}{k+1}$, $p=k+1$,
and $B_t(x)$ is a ball of radius $t$ centered at $x$. For a bounded domain $\Om$, we will set $R = 2 \, {\rm diam} \, \Om$. If
$\Om = \R^n$,  then $R= +\infty$; in this case we drop the superscript $R$ and use the notation
\begin{equation}\label{h-s-wolffinfty}
{\rm\bf W}_{\alpha,\, p} \mu (x)=\int_{0}^{+\infty}\Big[\frac{\mu(B_{t}(x))}{t^{n-\alpha p}
}\Big]
^{\frac{1}{p-1}}\frac{dt}{t}, \qquad x \in \R^n.
 \end{equation}

Potentials ${\rm\bf W}_{\alpha,\, p} \mu$ appeared in \cite{HW} in relation to Th. Wolff's inequality, and
L.~I. Hedberg's solution of the spectral synthesis problem for Sobolev spaces (see  \cite{AH}, Sections 4.5 and  9.1). Local estimates involving Wolff's potentials have been used extensively in quasilinear  and Hessian
equations, especially in relation to analogues of the Wiener criterion  \cite{KM}, \cite{L}, \cite{TW4}.

 We will use  {\it global} estimates of solutions to  Hessian equations  \cite{PV1} which effectively
invert the Hessian operator $F_k$  on $\R^n$ with zero boundary conditions at  $\infty$:
\begin{equation}\label{glob}
  C_{1} \, {\rm\bf W}_{\alpha, \, p} \, \mu (x)  \le  -u(x)  \le C_{2} \,
  {\rm\bf W}_{\alpha, \, p} \, \mu (x), \quad x \in \R^n,
 \end{equation}
where $\alpha =\frac {2k}{k+1}$, $p=k+1$, and $C_{1}$ and  $C_{2}$ are positive constants depending only on $n$ and $k$.
Here  $u$ is an entire solution
to $F_k[u] = \mu$  which is  $k$-convex, and vanishes at $\infty$.

Analogous global
estimates for solutions to the Dirichlet problem associated with
$F_k[u] = \mu$ in a bounded domain $\Omega$ are given by (\cite{PV1}, Theorem 2.2):
\begin{equation}\label{glob-om}
   C_{1} \, {\rm\bf W}^r_{\alpha, \, p} \, \mu (x)  \le   -u(x) \le C_{2} \,
  {\rm\bf W}^{R}_{\alpha, \, p} \, \mu (x), \qquad x \in \Om,
 \end{equation}
where $r = \frac 1 2 {\rm dist} \, (x, \partial \Om)$, $R= 2 \, {\rm diam} \, \Om$; as above, $\alpha =\frac {2k}{k+1}$, $p=k+1$, and $C_1$, $C_2$ are positive constants depending only on $n$ and $k$.

Another important tool established in Sec.~\ref{hess-duality-sec} is a fully nonlinear analogue of Schwarz's inequality which enables us to employ powerful duality arguments:
\begin{equation}\label{hoeld-hess}
\left \vert \int_\Om u \, F_k[v] \,  \right \vert  \le  \Big (\int_\Omega |u| \, F_k[u]  \Big)^{\frac 1{k+1}} \Big (\int_\Omega |v| \, F_k[v]  \Big)^{\frac k {k+1}},
 \end{equation}
where $u, v$ are $k$-convex functions with zero boundary values in $\Om$.

We observe
that the expression $\langle u, \, v \rangle = \int_\Om u \, F_k[v] $ on the left-hand side of the preceding inequality defines a non-commutative inner product (for $k\ge 2$)  on the cone of
$k$-convex functions which plays a role of mutual energy
for the corresponding Hessian measures.

This paper is  also concerned with  extensions of the Hessian Sobolev inequality  to a class of the so-called trace inequalities, with  general Borel measures
$d \omega$  in place of Lebesgue measure $dx$ on the left-hand side in
 (\ref{h-s-in}). Such inequalities have numerous applications in linear
and nonlinear PDE (see \cite{AH}, \cite{M}, \cite{PV2},  \cite{V1}).

We will give in Sec.~\ref{trace inequalities} a characterization of  the Hessian trace inequality
\begin{equation}\label{h-s-trace}
||u||_{L^{q} (\Om, \, d \omega)} \le C \, \Big (\int_\Omega |u| \, F_k[u] \, dx \Big)^{\frac 1{k+1}}.
  \end{equation}
We show that (\ref{h-s-trace}) holds for $q>k+1$ and $1\le k < \frac {n}{2}$ if
\begin{equation}\label{h-s-adams}
\omega(B) \le c \, |B|^{(1- \frac{2k}{n}) \frac{q}{k+1}},
  \end{equation}
 for all balls $B$.  This is an extension of the
well-known theorem of D. R. Adams in the linear case $k=1$ (see \cite{AH}, Sec. 7.2). In particular, if $\omega$ is Lebesgue measure on a hyperplane, (\ref{h-s-trace})  characterizes integrability properties of traces of $k$-convex functions  on lower dimensional cross-sections
 of $(k-1)$-convex bodies $\Om$.

In the more difficult case $q=k+1$, we will show that (\ref{h-s-trace}) holds if
\begin{equation}\label{h-s-mazya}
\omega(E) \le c \, \rm{cap}_k \, (E),
  \end{equation}
for every compact set $E\subset \Omega$. Here  $\rm{cap}_k \, (\cdot)$ is the Hessian capacity
introduced by Trudinger and Wang \cite{TW3}, and studied subsequently in \cite{L}. In \cite{PV1}, a complete
description of the Hessian capacity  is given in terms of the well-understood
fractional capacity $\rm{cap}_{\frac {2k}{k+1}, \, 2k} (\cdot)$
associated with the Sobolev space $W^{\frac {2k}{k+1}, \, 2k}(\Om)$.

In the linear case, conditions of the type  (\ref{h-s-mazya}) that characterize  admissible measures for trace inequalities were introduced by  V. G. Maz'ya, and  used extensively in the spectral theory of the Schr\"odinger operator (see \cite{M}).

We observe that (\ref{h-s-adams}) and (\ref{h-s-mazya}) are sharp, and in fact necessary for (\ref{h-s-trace}) to hold  if $\Omega=\R^n$, or if
$\omega$ is compactly supported in $\Omega$.

Other conditions equivalent to (\ref{h-s-mazya})
which do not use capacities are readily available. In particular, we obtain the following fully nonlinear version of the C. Fefferman and D. H. Phong inequality \cite{F}. Suppose that  $d \omega = w \, dx$, where
$w\ge 0$ is a weight such that
\begin{equation}\label{h-s-cond2}
\int_{B_R\cap \Om} w^{1 + \epsilon} \, dx \le c \, R^{n-2k(1+\epsilon)},
\end{equation}
for some $\epsilon >0$, and every ball $B_R$. Then (\ref{h-s-trace}) holds with $q=k+1$. \smallskip

The above results, including complete proofs of the Hessian Schwarz and Poincar\'e inequalities, were presented in the special session
``Nonlinear Elliptic Equations and Geometric Inequalities'' of the AMS meeting at Courant Institute in March 2008. They were
also announced at the Oberwolfach  Workshop ``Real Analysis, Harmonic Analysis and Applications'' in July, 2011(see  \cite{V2}).
 Further developments involving  the relationship between the Hessian energy and the fractional
Laplacian energy are established in \cite{FFV}, \cite{FV}.


\section{Hessian operators and $k$-convex functions}\label{intro}


 Let $\Om$ be an open set in $\R^n$, $n\geq 2$. Let $k=1,\ldots, n$. For
$u\in C^{2}(\Om)$, let
\begin{eqnarray*}
F_{k}[u]=S_{k}(\lambda(D^{2}u)),
\end{eqnarray*}
where $\lambda(D^{2}u)=(\lambda_{1},\ldots,\lambda_{n})$ is the spectrum of
the Hessian matrix  $D^{2}u$, and $S_{k}$ is the $k$-th elementary symmetric
function on $\R^n$, that is,
\begin{eqnarray*}
S_{k}(\lambda)=\sum_{1\leq i_{1}<\cdots<i_{k}\leq n}\lambda_{i_{1}}\cdots\lambda_{i_{k}}.
\end{eqnarray*}
 Equivalently,
\begin{eqnarray*}
F_{k}[u]=[D^{2}u]_{k},
\end{eqnarray*}
where $[A]_{k}$ stands for the
sum of the $k\times k$ principal
minors of a symmetric matrix $A$.

A function $u\in C^{2}_{{\rm loc}}(\Om)$ is called $k$-convex
if
\begin{eqnarray*}
F_{j}[u]\geq 0 \, \, {~\rm in~} \, \Om {~\rm \, \, for ~ all~} \, \, j=1,\dots, k.
\end{eqnarray*}
The notion  of $k$-convexity  can be defined for more general classes of functions   in terms of viscosity
solutions  \cite{TW2}, \cite{TW3}.  A function $u:
\Om\rightarrow [-\infty, \infty)$ is  said to be  $k$-convex in $\Om$ if it is upper semi-continuous, and $F_{k}[q]\geq 0$ for any quadratic polynomial $q$ such that  $u-q$ has
a local finite maximum in $\Om$.

The class of all $k$-convex
functions in $\Om$ which are not identically equal to $-\infty$ in each component of  $\Om$  will be denoted by $\Phi^{k}(\Om)$.
We observe  that $\Phi^{n}(\Om)\subset\Phi^{n-1}(\Om)\cdots\subset
\Phi^{1}(\Om)$  (see \cite{TW2}, \cite{TW3}. Here $\Phi^{1}(\Om)$ coincides with the set of all
subharmonic functions (with respect to the Laplacian) in $\Om$,  and $\Phi^{n}(\Om)$ is the set of functions
convex on each component of $\Om$.  By  $\Phi_0^{k}(\Om)$ we denote the class of all functions from $\Phi^{k}(\Om)$
with zero boundary values.

The theory of $k$-convex functions is based on the   weak continuity result  \cite{TW3}. It states that for each $u\in\Phi^{k}(\Om)$, there exists a nonnegative Borel measure
$\mu_{k}[u]$ in $\Om$ such that \\
{\rm (i)} $\mu_{k}[u]=F_{k}[u]$ for $u\in C^{2}(\Om)$, and\\
{\rm (ii)} if $\{u_{m}\}$ is a sequence in $\Phi^{k}(\Om)$ converging in $L^{1}_{\rm loc}(\Om)$
to a function $u\in\Phi^{k}(\Om)$, then the sequence of the corresponding measures
$\{\mu_{k}[u_{m}]\}$ converges weakly   to $\mu_{k}[u]$.

The measure $\mu=\mu_{k}[u]$ in the above theorem  is called the $k$-Hessian measure associated
with $u$. Property  (i)  justifies writing $F_{k}[u]$ in place of
$\mu_{k}[u]$ even when $u\in\Phi^{k}(\Om)$ does not belong to $C^{2}(\Om)$.

In what follows we will assume that the domain $\Om$ is uniformly $(k-1)$-convex, i.e.,  $H_j[\partial \Om] >0$ on $\partial \Om$ for $j=1,2, \ldots, k-1$. Here $H_j[\partial \Om]$ is the $j$-th mean curvature  defined by $H_j[\partial \Om] =S_j(\kappa_1, \kappa_2,
\ldots, \kappa_{n-1})$ where $S_j$ is the $j$-th elementary symmetric function of the principal curvatures $\kappa_1, \kappa_2, \ldots, \kappa_{n-1}$ of $\partial \Om$.

\section{Duality for Hessian integrals and Poincar\'e inequalities}\label{hess-duality-sec}

In this section we prove a duality theorem for Hessian integrals which will serve as a powerful tool
in our approach to Hessian Sobolev and Poincar\'e type inequalities. It can be regarded as a nonlinear version of the usual Schwarz inequality.

Let $u, v \in \Phi^k_0(\Om)$, where $k=1, 2, \ldots, n$ and let $\mu, \nu \in M^+(\Om)$ be the corresponding Hessian measures.
We introduce the following notation and terminology analogous
to the classical case $k=1$. By $ \mathcal{E}_k[\mu, \nu]$ we denote the mutual Hessian energy of $\mu$ and $\nu $ defined by
 \begin{equation}\label{hess-mutual-energy}
 \mathcal{E}_k [\mu, \nu] = \langle u, \, v \rangle =\int_{\Om} (-v) \,  F_k[u],
  \end{equation}
 where   $ \langle u, \, v \rangle$ is  the non-commutative (for $k \ge 2$) inner product on the cone of $k$-convex functions  mentioned
 in the Introduction.
Since $v \le 0$ by the maximum principle, and $F_k[u] \ge 0$, it follows that $\mathcal{E}_k [\mu, \nu] \ge 0$.  By $  \mathcal{E}_k [\mu] =    \mathcal{E}_k [\mu, \mu] $
we denote the Hessian integral
\begin{equation}\label{hess-integral-en}
 \mathcal{E}_k[\mu] =  \int_{\Om}  (-u) \,  F_k[u].
 \end{equation}
When $k=1$,
\begin{equation}\label{mutual-energy}
 \mathcal{E}_1 [\mu, \nu] = \int_{\Om}  G[\nu] \, d \mu =
 \int_{\Om}  G [\mu] \, d \nu,
  \end{equation}
  where $G [\mu] (x) = \int_\Om G(x,y) \, d \mu(y)$ is the Green
  potential of $\mu$, and $G(x,y)$ is the Green function of
  the Dirichlet Laplacian on $\Om$.  As is well known,
   \begin{equation}\label{schwarz}
  \mathcal{E}_1 [\mu, \nu]  = \int_\Om D G [\mu] \cdot D G [\nu] \, dx
  \le \mathcal{E}_1 [\mu]^{\frac 1 2} \mathcal{E}_1 [\nu]^{\frac 1 2},
  \end{equation}
  where
   \begin{equation*}
  \mathcal{E}_1 [\mu] =
   \int_{\Om}  G [\mu] \, d \mu =    \int_{\Om}  |D G [\mu]|^2 \, d x
     \end{equation*}
    denotes the electrostatic energy of $\mu$.

We now prove a nonlinear analogue of Schwarz's inequality (\ref{schwarz}) for Hessian integrals mentioned in the Introduction:
\begin{equation}\label{hess-hold}
 \langle u, \, v \rangle = \mathcal{E}_k[\mu, \nu] \le  \mathcal{E}_k[\mu]^{\frac {k} {k+1}}  \, \mathcal{E}_k[\nu]^{\frac {1} {k+1}}.
  \end{equation}
The proof  employs a convexity argument  that was used earlier
by X.-J. Wang in his proof of the Hessian Minkowski inequality
(\cite{W}, Theorem 5.1).
\begin{equation}\label{hess-mink}
\mathcal{E}_k [\mu+\nu]^{\frac 1 {k+1}} \le \mathcal{E}_k [\mu]^{\frac 1 {k+1}}  +  \mathcal{E}_k [\nu]^{\frac 1 {k+1}} .
\end{equation}
In the linear case $k=1$, (\ref{hess-hold}) and (\ref{hess-mink})  are known to be equivalent,
but for $k \ge 2$ the relationship between them is less obvious.

\begin{theorem}\label{hess-holder} Let $k=1, 2, \ldots, n$. Let  $\Om$ be a bounded uniformly $(k-1)$-convex domain with $C^2$ boundary. Let $u, v \in \Phi^k_0(\Om)$,
and let $\mu, \nu$ be the corresponding Hessian measures.
Then (\ref{hess-hold}) holds.
\end{theorem}
\begin{proof} We will use some standard notation and basic facts of the Hessian  theory. Let $r=(r_{ij})_{i,j=1}^n$ be a real symmetric matrix, and $\lambda[r]$ its eigenvalues. By
$S_k(\lambda[r])=[r]_k$ ($k=1,2, \ldots, n$) we denote the sum of all
the $k\times k$ principal minors of $r$. Let $S^{i j}_k[r]= \frac{\partial}{\partial r_{ij}} S_k(\lambda[r])$ ($i, j=1,2, \ldots, n$). Then (see \cite{W}, p. 29)
\begin{equation}\label{hess-sij}
 S_k(\lambda[r]) = \frac 1 k \sum_{i,j} r_{ij} S^{ij}_k [r].
   \end{equation}
   In particular, if $r=D^2 u$, then
   \begin{equation}\label{hess-Dij}
 S_k(\lambda[D^2u]) = \frac 1 k \sum_{i,j} u_{ij} S^{ij}_k [D^2u]= \frac 1 k \sum_{i,j} D_j (D_i u \, S^{ij}_k[D^2u]),
   \end{equation}
since
\begin{equation}\label{hess-null}
\sum_{j}  D_j (S^{ij}_k[D^2 u]) = 0,
   \end{equation}
   for every $i=1,2, \ldots, n$. Hence,
  \begin{equation}\label{hess-utv}
  S_k(\lambda[D^2(u+tv)]) =  \frac 1 k   \sum_{i, j}  (u_{ij} +t v_{ij}) \, S^{ij}_k [D^2(u+tv)].
  \end{equation}

   We will need the following formula   for the derivative of  $S_k(\lambda[D^2(u+tv)])$:
   \begin{equation}\label{hess-reilly}
  \frac{d}{dt} S_{k}(\lambda[D^2(u + t v)]) = \sum_{i, j} v_{ij} \, S^{ij}_k[D^2(u+tv)].
   \end{equation}
   The preceding equation is deduced using R.~C.~Reilly's identities, as in \cite{W} in the case of equation  \eqref{hess-sij}:
     \begin{equation}\label{reilly1}
   S_k(\lambda[D^2(u+tv)]) = \frac{1}{k!}  \underset{j_1, \ldots, j_k}{\sum_{i_1,  \ldots, i_k} } \delta_{j_1,  \ldots, j_k}^{i_1,  \ldots, i_k} \,
   \prod_{s=1}^k (u_{i_s j_s} +t v_{i_s j_s}),
     \end{equation}
    where $\delta_{j_1,  \ldots, j_k}^{i_1,  \ldots, i_k}$ is the generalized Kronecker delta. This yields (see equation (2.4) in \cite{W})
     \begin{equation*}\label{reilly2}
  S_k^{ij} [D^2(u+tv)] = \frac{1}{(k-1)!}  \underset{j_1, \ldots, j_{k-1}}{\sum_{i_1,  \ldots, i_{k-1}}}  \delta_{j_1,  \ldots, j_{k-1}, j}^{i_1,  \ldots, i_{k-1}, i} \, \prod_{s=1}^{k-1} (u_{i_s j_s} +t v_{i_s j_s}) .
      \end{equation*}
      Differentiating both sides of \eqref{reilly1} with respect to $t$,
    and invoking the preceding equation,  together with the symmetry properties of the Kronecker delta,
      we see that
$ \frac{d}{dt} S_{k}(\lambda[D^2(u + t v)]) $ equals
\begin{align*}
&\frac{1}{k!}  \underset{j_1, \ldots, j_k}{\sum_{i_1,  \ldots, i_k} } \delta_{j_1,  \ldots, j_k}^{i_1,  \ldots, i_k} \,  \sum_{m=1}^{k} v_{i_m j_m}
\underset{s\not=m}{\prod_{s=1}^k} (u_{i_s j_s} +t v_{i_s j_s})
   \\
&=\frac{1}{k!}  \sum_{m=1}^{k} \sum_{i_m, j_m} v_{i_m j_m} \underset{j_1, \ldots, j_{m-1}, j_{m+1}, \ldots, j_k}{\sum_{i_1,  \ldots, i_{m-1},
i_{m+1}, \ldots, i_k} } \delta_{j_1,  \ldots, j_k}^{i_1,  \ldots, i_k} \, \underset{s\not=m}{\prod_{s=1}^k} (u_{i_s j_s} +t v_{i_s j_s})
   \\
& = \frac{(k-1)!}{k!}  \sum_{m=1}^{k} \sum_{i_m, j_m} v_{i_m j_m}  \, S_k^{i_m j_m} (D^2[u+t v]) =  \sum_{i, j} v_{i j} \, S_k^{i j} (D^2[u+t v]).
\end{align*}
This completes the proof of \eqref{hess-reilly}.

For $t \in [0,1]$, let
\begin{equation}\label{convex}
  h(t) =  \int_{\Om} - (u +tv)\,  S_k(\lambda[D^2(u+tv)]) \, dx.
  \end{equation}
Using (\ref{hess-Dij})  and integrating by parts,  we get
  \begin{equation}\label{convex2}
  h(t) =  \frac 1 k  \int_{\Om}  \sum_{i, j} D_i(u +tv)\,  D_j(u +tv) \, S^{ij}_k [D^2(u+tv)] \, dx.
  \end{equation}
  Notice that the matrix $ \left ( S^{ij}_k[D^2(u+tv)]\right )$ is nonnegative since
  $u+t v$ is $k$-convex.  Differentiating both sides of
  (\ref{convex}) with respect to $t$, and integrating by parts using (\ref{hess-null})--(\ref{hess-reilly}),  we
  deduce (see \cite{W}, p. 37):
     \begin{equation}\label{hess-diffh}
h'(t)= ( k+ 1)
 \int_{\Om} (- v)\,  S_k(\lambda[D^2(u+tv)]) \, dx
    \end{equation}
  \begin{equation*} =
 \frac {k+1}{k}
 \int_{\Om} \sum_{i, j} D_i v \,  D_j(u+t v) \, S^{ij}_k [D^2(u+tv)] \, dx.
   \end{equation*}
  Differentiating further, we obtain
 \begin{equation*}
  h''(t) =   (k+1) \int_{\Om} (- v)\,  \frac{d}{dt} S_k(\lambda[D^2(u+tv)]) \, dx
  \end{equation*}
     \begin{equation*}
  = (k+1) \int_{\Om}  \sum_{i, j} D_i v \, D_j v \,  S^{ij}_k[D^2(u+tv)] \, dx.
   \end{equation*}
   From these calculations,  we deduce using Schwarz's inequality,
       \begin{equation*}
       h(t) h''(t)\ge \frac{k}{k+1} h'(t)^2, \quad 0\le t \le 1.
       \end{equation*}
    In other words, $h^{\frac 1 {k+1}}$ is convex on $[0,1]$ (\cite{W}, p. 37). This fact yields the Hessian Minkowski inequality (\ref{hess-mink}).
  We now observe that by  the convexity of  $h^{\frac 1 {k+1}}$  on $[0,1]$,
\begin{equation}\label{convex1}
 h(0)^{\frac 1 {k+1}} + \frac{d}{dt} h^{\frac 1 {k+1}}(0) \le h(1)^{\frac 1 {k+1}}.
 \end{equation}
 Notice that
 \begin{equation*}
\mathcal{E}_k [\mu, \nu] =  \int_{\Om}  (-v) \,  F_k[u] =
\frac 1 k \int_\Om \sum_{i,j} u_i v_j \, S^{ij}_k (D^2u) \, dx.
\end{equation*}
Thus by  (\ref{convex}) and (\ref{hess-diffh}),
\begin{equation*}
h(0) = \mathcal{E}_k [\mu], \quad
h'(0) = (k+1) \mathcal{E}_k[\mu, \nu], \quad h(1) = \mathcal{E}_k [\mu+\nu],
 \end{equation*}
and hence
\begin{equation*}
 \frac{d}{dt} h^{\frac 1 {k+1}}(0) = \frac 1 {k+1} h(0)^{\frac {-k} {k+1}}
h'(0) = \mathcal{E}_k [\mu]^{\frac {-k} {k+1}} \, \mathcal{E}_k[\mu, \nu].
\end{equation*}
Consequently by (\ref{convex1}) and (\ref{hess-mink}),
\begin{equation*}
\mathcal{E}_k [\mu]^{\frac 1 {k+1}}+\mathcal{E}_k [\mu]^{\frac {-k} {k+1}} \, \mathcal{E}_k[\mu, \nu] \le  \mathcal{E}_k [\mu+\nu]^{\frac 1 {k+1}} \le \mathcal{E}_k [\mu]^{\frac 1 {k+1}}  +   \mathcal{E}_k [\nu]^{\frac 1 {k+1}}.
\end{equation*}
Thus,
\begin{equation*}
\mathcal{E}_k [\mu]^{\frac {-k} {k+1}} \, \mathcal{E}_k[\mu, \nu]
\le   \mathcal{E}_k [\nu]^{\frac 1 {k+1}}.
 \end{equation*}
 This completes the proof of (\ref{hess-hold}).
\end{proof}

An immediate consequence of Theorem~\ref{hess-holder} is a global
integral inequality of Poincar\'e type
due to Trudinger and Wang \cite{TW3}. It is used in the proof of both local and global Wolff potential estimates obtained in \cite{L}, \cite{PV1},  which
play a crucial role in our proof of the Hessian Sobolev inequality and its extensions
given below.
\begin{corollary}\label{hess-poincare} Let  $\Om$ be as in Theorem~\ref{hess-holder}, and $k=1, 2, \ldots, n$. Let $u \in \Phi^k_0(\Om)\cap C^2(\overline{\Om})$.
Then
\begin{equation}\label{hess-poinc}
\int_\Om (-u) \, dx \le C \, \left ( \int_\Om (-u) \, F_k[u] \, dx \right)^{\frac 1 {k+1}},
  \end{equation}
where the best constant is
 \begin{equation}\label{hess-poinc-best}
C= \left ( \int_\Om (-w) \, dx \right)^{\frac {k}{k+1}}.
\end{equation}
 Here $w \in \Phi^k_0(\Om)\cap C^2(\overline{\Om})$ is the unique
solution to the Dirichlet problem $F_k[w]=1$; the corresponding minimizer is $u=w$.
\end{corollary}
 \begin{proof}
By the existence theorem  for Hessian equations \cite{CNS}, one can find $w\in \Phi^k_0(\Om)\cap C^2(\overline{\Om})$ such that
$F_k[w]=1$ in $\Om$. By Theorem~\ref{hess-holder},
\begin{equation*}
\int_\Om (-u) \, dx \le  \left ( \int_\Om (-u) \, F_k[u] \, dx \right)^{\frac 1 {k+1}}
\left ( \int_\Om (-w) \, dx \right)^{\frac {k} {k+1}}.
\end{equation*}
Clearly, $u=w$ gives equality in (\ref{hess-poinc}).
\end{proof}

If $B_R$ is a ball of radius $R$, then $w=c(|x|^2-R^2)$ where $c= \frac {1}{2} (C_n^k)^{-\frac{1}{k}}$ in the above proof, and  we arrive at the following corollary.
\begin{corollary}\label{hess-poincare1} If
$B_R$ is a ball of radius $R>0$, and $c$ is as above,
then
\begin{eqnarray*}
 \frac{1}{R^n} \int_{B_R} (-u) \, dx \le c \,  \left (\int_{B_1(0)} (1-|x|^2) \, dx \right)^{\frac {k}{k+1}}  \\
 \times \left ( \frac{1}{R^{n-2k}}\int_{B_R} (-u)  \, F_k[u] \, dx \right)^{\frac 1 {k+1}},
  \end{eqnarray*}
  for all $u \in  \Phi^k_0(B_R)\cap C^2(\overline{B_R})$.
\end{corollary}

More generally, we can obtain higher order Hessian Poincar\'e inequalities with sharp constants on balls.
\begin{corollary}\label{hess-poincare3} If
$B_R$ is a ball of radius $R>0$ and $k=1,2, \ldots, n$, then
\begin{equation}\label{hess-poincRk}
 \left ( \int_{B_R} (-u) \, F_{k-1}[u]  \, dx \right)^{\frac 1 k}\le C \, \left ( \int_{B_R}  (-u) \, F_k[u] \, dx \right)^{\frac 1 {k+1}},
  \end{equation}
    for all $u \in \Phi^k_0(B_R)\cap C^2(\overline{B_R})$, where the best constant
   \begin{equation*}
    C = \left ( \int_{B_R}  (-w) \, F_k[w] \, dx \right)^{\frac {1} {k(k+1)}}
    \end{equation*}
    is attained when $u=w$, where $w= c (|x|^2 - R^2)$, $c = \frac{k}{2(n-k+1)}$.
\end{corollary}

\begin{proof}
We observe that  $w= c (|x|^2 - R^2)$ where $c = \frac{k}{2(n-k+1)}$   is a convex solution to the equation
$F_k[w]=F_{k-1}[w]$ such that  $w=0$ on $\partial B_R$.  Then using the divergence form of $F_k[u]$ \eqref{hess-Dij},
and integrating by parts, we obtain,
\begin{eqnarray}\label{by-parts}
 \int_{B_R} (-w) \, F_k[u] \, dx & = \frac{1}{k} \int_{B_R} \sum_{i, j} S^{ij}_k [D^2 u] \, (-u)   \, w_{ij}  \, dx
 \\ & = \frac{2c}{k}  \int_{B_R} \sum_{i} S^{ii}_k [D^2 u] \, (-u) \, dx \\ & = \frac{2c(n-k+1)}{k} \int_{B_R} (-u) \, F_{k-1} [u] \, dx.  \label{by-parts2}
  \end{eqnarray}
  In the last line we used the identity (see \cite{W}, Proposition 2.2)
  \begin{equation*}
  \sum_{i=1}^n S^{ii}_k [D^2 u] = (n-k+1) \, S_{k-1}[D^2 u].
  \end{equation*}

  From \eqref{by-parts}--\eqref{by-parts2}, applying Theorem~\ref{hess-holder} to estimate the left-hand side,
we deduce
\begin{eqnarray*}
\left ( \int_{B_R} (-u) \, F_{k-1} [u] \, dx \right)^{\frac 1 k} \le \left ( \frac{2c(n-k+1)}{k}\right)^{-\frac 1 k} \\ \times \left ( \int_{B_R} (-u) \, F_{k} [u] \, dx \right)^{\frac 1 {k+1}} \left ( \int_{B_R} (-w) \, F_{k} [w] \, dx \right)^{\frac 1 {k(k+1)}}.
  \end{eqnarray*}
  Since $F_k[w]=F_{k-1}[w]$, it follows
  \begin{eqnarray*}
  \left ( \frac{2c(n-k+1)}{k}\right)^{-\frac 1 k} \left ( \int_{B_R} (-w) \, F_{k} [w] \, dx \right)^{\frac 1 {k(k+1)}} \\=
  \frac{ \left ( \int_{B_R} (-w) \, F_{k-1} [w] \, dx \right)^{\frac 1 {k}} }{ \left ( \int_{B_R} (-w) \, F_{k} [w] \, dx \right)^{\frac 1 {k+1}}}.
   \end{eqnarray*}
  Thus, setting $u=w$ gives equality in (\ref{hess-poincRk}).
 \end{proof}

\begin{remark} Hessian Poincar\'e inequalities (\ref{hess-poincRk})  for all convex domains $\Om$ in place
of balls, and $k=1, 2, \ldots, n$, with sharp constant, can be deduced in a similar way.
\end{remark}

We next give a simple proof of the  higher order Hessian Poincar\'e inequalities of Trudinger and Wang \cite{TW1} for general $(k-1)$-convex domains, with explicit constants.
\begin{corollary}\label{hess-poincare-gen} Let  $\Om$ be as in Theorem~\ref{hess-holder},  and $2 \le l <k$, $k=1, 2, \ldots, n$.
Let $u \in \Phi^k_0(\Om)\cap C^2(\overline{\Om})$.
Then
\begin{equation}\label{hess-poinc-g}
\left ( \int_\Om (-u) \, F_l[u] \, dx \right)^{\frac 1 {l+1}} \le C \, \left ( \int_\Om (-u) \, F_k[u] \, dx \right)^{\frac 1 {k+1}},
  \end{equation}
where
$C= \frac {k^k}{l^l \, (k-l)^{k-l}} \left ( \int_\Om (-w) \, F_k[w] \, dx \right)^{\frac {k-l} {(l+1)(k+1)}}.$
 \end{corollary}
 \begin{proof}
 According to the existence theory for quotient Hessian equations  \cite{T1}, there is a unique solution  $w \in \Phi^k_0(\Om)\cap C^2(\overline{\Om})$ to the Dirichlet problem $F_k[w]=F_l[w]$. Then by the concavity in $D^2[u]$ of the operators $(F_l[u])^{\frac 1 {l}}$ and
$\left(\frac{F_k[u]}{F_l[u]}\right)^{\frac{1}{k-l}}$
  on the cone of $k$-convex functions  (see \cite{CNS}, \cite{T1}), it follows
\begin{equation*}
\int_{\Om} (-u) \, F_{k} [u+w] \, dx   =  \int_{\Om} (-u) \, F_{l} [u+w] \,  \frac{F_k[u+w]}{F_l[u+w]} \, dx
  \end{equation*}
 \begin{equation*}
\ge \int_{\Om} (-u) \, F_{l} [u]  \,  \frac{F_k[w]}{F_l[w]} dx  = \int_{\Om} (-u) \, F_{l} [u] \,  dx .
  \end{equation*}
Applying Theorem~\ref{hess-holder} to estimate the left-hand side, we deduce
\begin{equation*}
 \int_{\Om} (-u) \, F_{l} [u] \,  dx  \le \int_{\Om} (-u) \, F_{k} [u+w] \, dx
 \end{equation*}
 \begin{equation*}
  \le \left ( \int_\Om (-u) \, F_k[u] \, dx \right)^{\frac 1 {k+1}}
 \left ( \int_\Om -(u+w) \, F_k[u+w] \, dx \right)^{\frac k {k+1}}.
\end{equation*}
Using the Hessian Minkowski inequality (\ref{hess-mink}) to estimate
the second factor on the right-hand side, we obtain
\begin{equation*}
 \int_{\Om} (-u) \, F_{l} [u] \,  dx  \le \left ( \int_\Om (-u) \, F_k[u] \, dx \right)^{\frac 1 {k+1}}
 \end{equation*}
 \begin{equation*}
  \times \left [  \left ( \int_\Om (-u) \, F_k[u] \, dx \right)^{\frac 1 {k+1}}
 +  \left ( \int_\Om (-w) \, F_k[w] \, dx \right)^{\frac 1 {k+1}} \right]^{k}.
 \end{equation*}
 Replacing $u$ by $cu$, where $c>0$, in the preceding inequality and
 minimizing   over $c$ yields
(\ref{hess-poinc-g}).
 \end{proof}


\section{Hessian Sobolev and trace inequalities on bounded
domains}\label{trace inequalities}


In this section we give a new proof of the Hessian Sobolev inequality \cite{W},
and extend it to trace inequalities on a bounded $(k-1)$-convex domain $\Om \subset \R^n$. We will use duality of Hessian integrals developed in Sec.~\ref{hess-duality-sec}, and the classical Sobolev inequality for fractional integrals.

\begin{theorem}\label{hess-sobolev}  Let   $1\le k< \frac n 2$ and  $0<q \le \frac{n(k+1)}{n-2k}$. Let  $\Om$ be a bounded uniformly $(k-1)$-convex domain with $C^2$ boundary.
Then, for $u\in \Phi^k_0(\Om)\cap C^2(\overline{\Om})$,
\begin{equation}\label{hess-sobol}
\left ( \int_\Om |u|^q \, dx \right )^{\frac 1 q} \le C \, \left ( \int_\Om |u| \, F_k[u] \, dx \right)^{\frac 1 {k+1}},
  \end{equation}
where $C$ is a positive constant that depends only on $n, k$, and $q$.
\end{theorem}

\begin{remark}
In the special case $q=1$, inequality
\eqref{hess-sobol} was proved in Corollary~\ref{hess-poincare} with the best constant which depends on $\Omega$. The best constant in \eqref{hess-sobolev}
is known  for $\Om = \R^n$ and $q=\frac{n(k+1)}{n-2k}$ (see \cite{W}).
\end{remark}

\begin{proof}
Since $\Om$ is bounded, we can assume that  $q= \frac{n(k+1)}{n-2k}$ so that $q>1$. By the maximum principle, $u \le 0$ in $\Omega$.
Clearly, $C^\infty_0(\Om)$ is dense in $L^{q'}(\Om)$ where $\frac 1 q + \frac 1 {q'}=1$.
Hence by duality, (\ref{hess-sobol}) is equivalent to
\begin{equation}\label{hess-sobol-dual}
\int_\Om (-u) \, w \,  dx  \le C \, \left ( \int_\Om (-u) \, F_k[u] \, dx \right)^{\frac 1 {k+1}} ||w||_{L^{q'}(\Om)},
  \end{equation}
where $C$ does not depend on $w \in L^{q'}(\Om) \cap C^\infty_0(\Om)$, $w \ge 0$. Denote
by $v\in  \Phi^k_0(\Om) \cap  C^2 (\overline{\Om})$ the unique solution to the equation  $F_k[v]=w$. Then by Theorem~\ref{hess-holder},
\begin{equation}\label{s-dual}
\int_\Om (-u) \, w \,  dx  \le C \, \left ( \int_\Om (-u) \, F_k[u] \, dx \right)^{\frac 1 {k+1}} \left ( \int_\Om (-v) \, F_k[v] \, dx \right)^{\frac k {k+1}}.
  \end{equation}
We now invoke a global estimate of the solution $v$ to the Hessian equation $F_k[v]=w$  by means of Wolff's potential defined by
(\ref{h-s-wolffp}) (see \cite{PV1}, Theorem 2.2):
\begin{equation}\label {wolff-est}
-v(x) \le C \, {\rm\bf W}^{2{\rm diam}(\Om)}_{\alpha, p} w (x), \quad
x \in \Om,
  \end{equation}
  where $\alpha = \frac{2k}{k+1}$ and $p=k+1$, and $C$ depends only on $k$ and $n$.
Hence,
\begin{equation*}
\int_\Om (-v) \, F_k[v] \, dx \le C \int_\Om {\rm\bf W}^{2{\rm diam}(\Om)}_{\alpha, p} w \, w \, dx.
  \end{equation*}
  The Wolff potential ${\rm\bf W}^{2{\rm diam}(\Om)}_{\alpha, p} w$  on the right-hand side is obviously bounded by  ${\rm\bf W}_{\alpha, p} w$
   which
  corresponds to $\Om = \R^n$, see (\ref{h-s-wolffinfty}).  Here we extend  $w$  to $\R^n$ so that $w=0$ outside $\Om$.

We next replace ${\rm\bf W}_{\alpha,\, p} w$
with the Havin--Maz'ya potential $I_\alpha (I_\alpha w)^{\frac{1}{p-1}}$ using the pointwise estimate (\cite{M}, Sec. 10.4.2)
\begin{equation*}
{\rm\bf W}_{\alpha,p} w (x) \le C \, I_\alpha (I_\alpha w)^{\frac{1}{p-1}} (x), \qquad x \in \R^n,
  \end{equation*}
  where $I_\alpha= (-\Delta)^{-\frac \alpha 2}$ is a Riesz potential of order $\alpha$, and $C$ is a positive constant depending only on
$\alpha, p, n$.
We deduce
\begin{equation*}
\int_\Om (-v) \, F_k[v] \, dx \le C \int_\Om {\rm\bf W}_{\alpha, p} w \, w \, dx
 \end{equation*}
\begin{equation*}
\le C \int_\Om I_\alpha
(I_\alpha w)^{\frac{1}{p-1}} w \, dx = C \, \int_{\R^n} (I_\alpha w)^{\frac{p}{p-1}} \, dx.
 \end{equation*}
 Now  applying the Sobolev inequality for fractional integrals of order $\alpha= \frac{2k}{k+1}$ (see, e.g., \cite{M}, Sec. 8.3), we obtain
 \begin{equation*}
\int_\Om (-v) \, F_k[v] \, dx \le C \, \int_{\R^n} (I_\alpha w)^{\frac{p}{p-1}} \, dx \le C \, ||w||^{\frac{p}{p-1}}_{L^{q'}(\Om)},
 \end{equation*}
where $\frac{p}{p-1}=\frac{k+1}{k}$.
Combining the preceding inequality with (\ref{s-dual}), we
arrive at  (\ref{hess-sobol-dual}).
\end{proof}

By inspecting the proof of Theorem~\ref{hess-sobolev}, we
deduce the following inequality for Hessian integrals  which  can be regarded as a dual form of (\ref{hess-sobol}).
\begin{corollary}\label{cor-dual}
Let $1\le k < \frac n 2$. Let $v \in \Phi^k_0(\Om)\cap C^2(\overline{\Om})$. Then
 \begin{equation}\label{dual-hess-sob}
 \left ( \int_\Om (-v) \, F_k[v] \, dx \right )^{\frac{1}{k+1}} \le
 C \, \left ( \int_\Om (F_k[v])^{q'} \, dx \right )^{\frac{1}{q'}},
  \end{equation}
  where $q'=\frac{n(k+1)}{(n+2)k}$ is dual to
  the Hessian Sobolev exponent $q=\frac{n(k+1)}{n-2k}$.
  \end{corollary}

 Theorem~\ref{hess-sobolev} is a special case of the following trace theorem.
\begin{theorem}\label{H-S-Om}
Let $1\leq k<\frac{n}{2}$, and $q>k+1$.
Let $\omega$ be a positive Borel measure on a bounded uniformly $(k-1)$-convex domain $\Om$ with $C^2$ boundary.
Suppose that $\omega$  obeys
\begin{equation}\label{cond3}
\varkappa(\omega) = \sup_{B} \frac{\omega(B\cap \Om)} { |B|^{(1-\frac{2k}{n})\frac{q}{k+1}}}
< \infty,
\end{equation}
where the supremum is taken over all balls $B$ in $\R^n$. Then, for every
 $u\in\Phi_0(\Om)\cap C^2(\overline{\Om})$,
\begin{equation}\label{h-s-in-om}
 ||u||_{L^{q} (\Om, \, d \omega)} \le C \, \varkappa(\omega)^{\frac 1 {q}} \, \Big (\int_{\Om} |u| \, F_k[u] \, dx \Big)^{\frac 1{k+1}},
\end{equation}
where the constant $C$ depends   on $ n, k, q$, and $\Om$.
\end{theorem}

\begin{remark} If $k=\frac n 2$ for even $n$, then \rm{(\ref {h-s-in-om})} holds if
\begin{equation}\label{cond4}
\varkappa(\omega) = \sup_{B} \frac{ \omega(B\cap \Om)} { \left ( \log \frac{2 }{|B|}\right )^{\frac{qk}{k+1}}} < \infty,
\end{equation}
for every ball $B$ such that $|B| \le 1$.
\end{remark}

\begin{proof}  Let $\alpha = \frac{2k}{k+1} $, $p=k+1$, where $k=1, 2, \ldots, [ \frac {n}{2}]$. Let $\mu=\mu_k[u]$ be the
Hessian measure associated with $u$ supported in $\overline{\Om}$. As in the proof of Theorem~\ref{hess-sobolev}, by duality  (\ref{h-s-in-om}) is equivalent to
\begin{equation}\label{hess-sobol-dual1}
\int_\Om (-u) \, w \,  d \om  \le C \, \varkappa(\omega)^{\frac 1 {q}} \, \left ( \int_\Om (-u) \, F_k[u] \, dx \right)^{\frac 1 {k+1}} ||w||_{L^{q'}(\Om, \, \om)},
  \end{equation}
where $C$ does not depend on $w \in L^{q'}(\Om, \, \om) \cap C^\infty_0(\Om)$, $w \ge 0$. Without loss of generality we may assume that $\omega$ is compactly
supported in $\Om$. Denote
by $v\in  \Phi^k_0(\Om)$ a solution to the equation  $F_k[v]=\mu$, where $d \mu = w \, d \om$.
Then by Theorem~\ref{hess-holder} and the global Wolff potential estimate
(\ref{wolff-est}), it suffices  to prove the inequality
\begin{equation}\label{integr-ineq1}
\int_{\Om} {\rm\bf W}_{\alpha,\, p} \mu \, d \mu \leq
C \, \varkappa(\omega)^{\frac{p'}{q}} \,  ||w||^{p'}_{L^{q'}(\Om, \, \om)},
 \end{equation}
 As in the proof of the Hessian Sobolev inequality above,
$$
\int_{\Om} {\rm\bf W}_{\alpha,\, p} \mu  \, d \mu \le C \, ||I_\alpha \mu||^{p'}_{L^{p'} (\R^n)},
$$
where $C$ depends only on $n$, $p$, $\alpha$. It remains to notice that
$$
||I_\alpha \mu||_{L^{p'} (\R^n)}= ||I_\alpha (w \, d \om)||_{L^{p'} (\R^n)} \le C \, \varkappa(\omega)^{\frac 1 {q}} \, ||w||_{L^{q'}(\Om, \, \om)},
 $$
 since by duality it is equivalent to D. Adams's inequality for Riesz potentials $I_\alpha$ (\cite{AH}, Sec. 7.2):
 $$
 ||I_\alpha f||_{L^{q} (\R^n, \, \om)} \le C \,  \varkappa(\omega)^{\frac 1 {q}} \, ||f||_{L^{p}(\R^n)},
 $$
for all $f \in L^p(\R^n)$.
\end{proof}

By using the same argument as in the preceding theorem, but  with a capacitary
characterization of the trace inequality for fractional integrals  (\cite{AH}, Sec. 7.2; \cite{M}, Sec. 11.3) in place of
D. Adams's theorem, we arrive at a
characterization of  (\ref{h-s-in-om}) in the case $q=k+1$ in terms of the classical Riesz capacity ${\rm cap}_{\alpha, \, p} (\cdot)$.
\begin{theorem}\label{h-s1-om}
 Let $q=k+1$ and  $1\leq k\le \frac{n}{2}$.
Let $\omega$ and  $\Om$ be as in Theorem~\ref{H-S-Om}.
Suppose that $\omega$  obeys
\begin{equation}\label{cond1-om}
 \omega(E) \le c \, {\rm cap}_{\frac{2k}{k+1}, \, 2k}  (E),
\end{equation}
for every compact set  $E\subset \Om$. Then, for every $u\in\Phi_0(\Om)\cap C^2(\overline{\Om})$,
\begin{equation}\label{h-s-in1-om}
 ||u||_{L^{q} (\Om, \, d \omega)} \le C  \, \Big (\int_\Omega |u| \, F_k[u] \, dx \Big)^{\frac 1{k+1}},
\end{equation}
where the constant $C$ depends  only on $ n, k$, and $c$.
\end{theorem}

The Hessian version of the  Fefferman--Phong inequality  stated in the Introduction is a consequence of
 Theorem~\ref{h-s1-om}.
  \begin{corollary}\label{fef-phong-om}
Suppose $q= k+1$ and $1\le k < \frac n 2$. Under the assumptions of Theorem~\ref{h-s1-om}, suppose that  $d \omega = w \, dx$, where
$w\ge 0$ is a weight on $\Om$ such that
\begin{equation}\label{cond2-om}
\int_{B_R\cap \Om} w^{1 + \epsilon} \, dx \le C \, R^{n-2k(1+\epsilon)},
\end{equation}
for some $\epsilon >0$, and every ball $B_R$.  Then
(\ref{h-s-in1-om}) holds.
\end{corollary}
\begin{proof} Suppose (\ref{cond2-om}) holds. By the
classical version of the  Fefferman--Phong inequality for fractional
integrals the operator $I_\alpha: \, L^q(\R^n) \to L^q(\Om, \, \om)$ is bounded for $q=k+1$, which yields (\ref{cond1-om})  (see, e.g., \cite{V1}). Then by Theorem~\ref{h-s1-om} inequality
(\ref{h-s-in1-om}) holds.  \end{proof}

\begin{remark} Other Hessian inequalities, in particular
(\ref{h-s-in1-om}) in the case $q<k+1$, are easily characterized using
the same proof as in Theorem~\ref{H-S}, along with the corresponding trace inequalities for fractional integrals (\cite{V1}, Theorem 1.13).
\end{remark}


\section{Hessian trace inequalities on $\R^n$}\label{hess-sobol-proofs}


In this section we prove general trace inequalities on the entire Euclidean  space $\R^n$. In this case there is no need to use Hessian duality. Our main tools are global estimates of solutions to Hessian equations (\ref{glob}) and  Wolff's inequality \cite{HW}.
A simplified proof of Wolff's inequality, along with a thorough discussion of related facts of nonlinear potential theory,
 can be found in \cite{AH}, Sections 4.5 and 9.3.

Suppose $q = \frac{n(k+1)}{n-2k}$, and $k$ is an integer such that $1\le k < \frac n 2$. If $u$ is a
$k$-convex $C^2$-function on  $\R^n$ which vanishes at infinity, then the
 Hessian Sobolev inequality
\begin{equation}\label{h-s-entire}
||u||_{L^{q} (\R^n, \, dx)} \le C \, \Big (\int_{\R^n} |u| \, F_k[u] \, dx \Big)^{\frac 1{k+1}}
 \end{equation}
holds. The preceding statement is a special case of the following theorem.
\begin{theorem}\label{H-S}
Let $\omega$ be a positive Borel measure on $\R^n$, $1\leq k<\frac{n}{2}$, and $q>k+1$.
Suppose that $\omega$  obeys
\begin{equation}\label{cond}
\varkappa(\omega) = \sup_{B} \frac{\omega(B)} { |B|^{(1-\frac{2k}{n})\frac{q}{k+1}}}
< \infty,
\end{equation}
where the supremum is taken over all balls $B$ in $\R^n$. Then, for every
$k$-convex $C^2$-function $u$ on  $\R^n$ such that $\sup_{x \in \R^n} \, u(x) =0$,
\begin{equation}\label{H-S-ineq}
 ||u||_{L^{q} (\R^n, \, d \omega)} \le C \, \varkappa(\omega)^{\frac 1 {q}} \, \Big (\int_{\R^n} |u| \, F_k[u] \, dx \Big)^{\frac 1{k+1}},
\end{equation}
where the constant $C$ depends  only on $ n, k, q$.

Conversely, if {\rm (\ref{H-S-ineq})} is valid for
$q>k+1$ and $1 \le k<\frac{n}{2}$ , then {\rm (\ref{cond})} holds. If $k=\frac n 2$ for even $n$, then \rm{(\ref {H-S-ineq})} holds only if $\omega=0$. \end{theorem}

\begin{proof}  Let $\alpha = \frac{2k}{k+1} $, $p=k+1$, where $k=1, 2, \ldots, [\frac n 2]$. Let $\mu=\mu_k[u]$ be the
Hessian measure associated with $u$. By the global Wolff potential estimate
(\ref{glob}), it suffices  to prove the inequality
\begin{equation}\label{integr-ineq}
\int_{\R^n} \left({\rm\bf W}_{\alpha,\, p}\mu\right)^q \, d \omega \leq
C \, \varkappa(\omega) \, E_k [\mu]^{\frac{q}{p}},
 \end{equation}
where $\omega$, $\mu$ are positive Borel measures on $\R^n$, and the Wolff energy
$E_k [\mu]$ is defined by
\begin{equation}\label{energy}
E_k [\mu] = \int_{\R^n} {\rm\bf W}_{\alpha,\, p}\mu \, d \mu.
\end{equation}
It was shown in \cite{PV1} that $E_k [\mu]$ is equivalent to the Hessian energy
\begin{equation}\label{hes-energy}
\mathcal{E}_k [\mu] =  \int_{\R^n} (-u) \, F_k[u].
\end{equation}
 Let $
g = (I_\alpha \, \mu)^{\frac 1{p-1}},$
where $I_\alpha=(-\Delta)^{-\frac \alpha 2}$ is a Riesz potential of order $\alpha$ on $\R^n$. By Wolff's inequality \cite{HW} (see also \cite{AH}, Sec. 4.5),
$$
\mathcal{E}_k [\mu] = \int_{\R^n} {\rm\bf W}_{\alpha,\, p}\mu \, d \mu \ge
C(p, \alpha, n) \, \int g^p \, dx.
$$
On the other hand, consider the Havin--Maz'ya nonlinear potential
$$ {\rm\bf U}_{\alpha,\, p}\mu=I_\alpha \, g = I_\alpha \,(I_\alpha \, \mu)^{\frac 1{p-1}}.$$
 It is well-known
 that  ${\rm\bf U}_{\alpha,\, p}\mu$ dominates ${\rm\bf W}_{\alpha,\, p}\mu$  (see, e.g., \cite{M}, Sec. 10.4.2), i.e.,
$$ {\rm\bf W}_{\alpha,\, p}\mu (x) \le C \, {\rm\bf U}_{\alpha,\, p}\mu(x)
= C \, I_\alpha \, g(x), \qquad x \in \R^n,
$$
where $C$ depend only on $n$, $\alpha$, and $p$. Thus,  (\ref{integr-ineq})  reduces to
$$
||I_\alpha g ||_{L^q(d \omega)} \leq C \, \varkappa(\omega) \, ||g||_{L^p(dx)},
$$
which follows from D. Adams's inequality for Riesz potentials (\cite{AH}, Sec. 7.2), or the classical Sobolev inequality for fractional integrals in case $d \omega$ is Lebesgue measure.

We now prove the necessity of  condition (\ref{cond}). We fix a compact set
$E \subset \R^n$, and denote by $\mu$ the equilibrium measure on $E$ in the sense
of the $k$-Hessian capacity (see \cite{L}). Denote by $u$ a $k$-convex solution
 to the equation $F_k[u] = \mu$ that vanishes at $\infty$. Then $u \ge 1$ on $E$, and the Hessian energy
of $\mu$ is equivalent to the Hessian capacity $\rm{cap}_k \, (E)$. Thus,
$\omega(E) \le C \, \rm{cap}_k \, (E).$
If $E$ is a ball $B$ of radius $r$, then $\rm{cap}_k \, (B) = c \, r^{n-2k}$,
if $k< \frac n 2$, and (\ref{cond}) holds; if $k= \frac{n}{2}$ then
$\rm{cap}_k \, (B) = 0$ and consequently $\omega(B)=0$ for every ball $B$ (see \cite{L}, \cite{PV1}). \end{proof}

\begin{remark} As in the case of bounded domains treated in
Sec.~\ref{trace inequalities}, for $q=k+1$
there is a similar
characterization of  (\ref{H-S-ineq})  in terms of Riesz capacities. Noncapacitary characterizations of (\ref{H-S-ineq}) for $q=k+1$, including
the Fefferman--Phong condition (\ref{h-s-cond2}), along with  similar
results for $q<k+1$, are easily deduced as well
 using the argument employed above combined with the corresponding results for fractional integrals (see \cite{V1}).
 \end{remark}

%
%

\section{Remarks on the proof of the Hessian Sobolev inequality}\label{conclusion}

%
%

It is worth observing that our approach to the Hessian Sobolev inequality  requires  certain modifications
of some proofs  in the theory of Hessian equations (\cite{TW3}, \cite{L}, \cite{PV1}). They have to be redone in order
to avoid an apparent circle argument since the Hessian Sobolev inequality was
referred to  in the original proofs.

Notice that the proofs of the global existence theorems  \cite{CNS}, \cite{T1} along with the interior gradient estimates \cite{T3}, and local integral estimates  of the type \cite{TW3}
\begin{equation}\label{concl-loc}
   \left(  \frac{1}{R^{n-2k}} \int_{B_{\frac {9}{10} R}} F_k[u] \, dx  \right)^{\frac 1 k} \le \frac{C}{R^n} \int_{B_R} |u| \, dx,
     \end{equation}
 for $u \in \Phi^k(B_R)$, $u \le 0$, do not  use  the global inequality (\ref{h-s-in}). The proofs of the comparison principle \cite{TW2}, and the weak
continuity theorem, which is essentially local  (\cite{TW3}, \cite{TW4}), do
not require the use of (\ref{h-s-in}) either.

Some changes are needed in the proofs of the pointwise  Wolff potential estimates
\cite{L}, \cite{PV1} where the full strength of the Hessian Sobolev inequality is not necessary. What is actually
used (see, e.g.,  the proof of estimate (2.19) in \cite{L}, p. 13) is a version of
the Poincar\'e inequality for $u\in \Phi^k_0(B_R)$ on balls $B_R$:
\begin{equation}\label{concl-glob}
\frac {1}{R^n}\int_{B_R} |u| \, dx \le C \, \left ( \frac{1}{R^{n-2k}}\int_{B_R} |u| \, F_k[u] \, dx \right)^{\frac 1 {k+1}}.
  \end{equation}
The preceding estimate  was deduced above (Corollary~\ref{hess-poincare1})
using a straightforward duality argument.

\end{document}